\documentclass{article}

\usepackage{arxiv}

\usepackage[utf8]{inputenc} 
\usepackage[T1]{fontenc}    
\usepackage[unicode]{hyperref}
\usepackage{url}            
\usepackage{booktabs}       
\usepackage{amsfonts}       
\usepackage{nicefrac}       
\usepackage{microtype}      
\usepackage{lipsum}		
\usepackage{graphicx}
\usepackage{natbib}
\usepackage{doi}
\usepackage{amssymb}
\usepackage{array}
\usepackage{latexsym}
\usepackage{enumerate}
\usepackage{amsmath}
\usepackage{amsfonts}
\usepackage{amsthm}
\usepackage{mathtools}
\usepackage{xcolor}
\usepackage[english]{babel}

\newtheorem{theorem}{Theorem}[section]
\newtheorem{proposition}[theorem]{Proposition}
\newtheorem{lemma}[theorem]{Lemma}

\newtheorem{remark}[theorem]{Remark}

\newtheorem{example}[theorem]{Construction}

\def\dG{\overrightarrow{G}}
\def\dO{\overrightarrow{O}}
 \newcommand{\aut}{\rm{Aut}}

\title{Automorphisms and quotients of 2-colored quasi best match graphs}

\date{} 					

\author{ \href{https://orcid.org/0000-0002-7334-669X}{\includegraphics[scale=0.06]{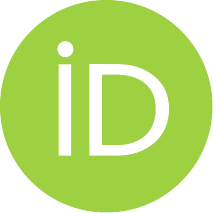}\hspace{1mm}Annachiara.~Korchmaros}\thanks{https://akorchmaros.com/} \\
	Universit{\"a}t Leipzig\\ Leipzig,      Germany\\
	\texttt{annachiara.korchmaros@uni-leipzig.de} \\
}

\renewcommand{\shorttitle}{\textit{arXiv} Template}

\hypersetup{
pdftitle={Automorphisms and quotients of 2-colored quasi best match graphs},
pdfauthor={Annachiara.~Korchmaros},
pdfkeywords={best match graphs, group of automorphisms, bipartite graphs, phylogenetics},
}

\begin{document}
\maketitle

\begin{abstract}
2-colored quasi best match graphs (2-qBMGs) are directed graphs that arose in phylogenetics. Investigations of 2-qBMGs have mostly focused on computational issues. However, 2-qBMGs also have relevant properties for structural graph theory; in particular, their undirected underlying graph is free
from induced paths and cycles of size at least $6$. In this paper, results on the structure of the automorphism groups of 2-qBMGs are obtained, which shows how to construct 2-qBMGs with large automorphism groups.
\end{abstract}

\keywords{ Best match graphs \and Group of automorphisms \and Bipartite graphs \and Phylogenetics}

\section{Introduction}
Automorphisms (also called symmetries) of geometric structures are elaborated, widely used technical tools, especially in topology and algebraic geometry. In the last three decades, they also appeared in numerous investigations of combinatorial structures, such as finite geometries and graphs. Motivation for the study of graph automorphisms came from group theory since many important groups were known to enjoy a very useful permutation representation as an automorphism group of a particular graph whose specific features provide a complete description of the structure of the group in terms of actions on vertices, edges, and certain subgraphs. 
 
This was the starting point for ongoing investigations aimed at characterizing families of groups with prescribed action on some directed or undirected graphs, involving primitivity on vertices, transitivity on edges and arcs, or flag-transitivity.  
Conversely, one can ask to describe the possibilities for the structures and actions of the automorphism groups of graphs belonging to a given family of directed or undirected graphs. A considerable amount of investigations have been done on this problem, particularly on Cayley graphs, cubic graphs, and tournaments, following the pioneering work in the 1980s by Alspach, Maru\v{s}i\v{c}, Moon, and others. For the rich literature on this topic, the reader is referred to the recent monograph~\cite{dobson2022symmetry}.  
 
In this paper, we focus on the family of 2-qBMGs (two-color quasi-best match graphs), which arose from evolution theory~\cite{geiss2019best,schaller2021corrigendum,manuela2020reciprocal,Schaller_2021, schaller2021heuristic,Ram_rez_Rafael_2024,Hellmuth_2024}. A purely graph-theoretic characterization of 2-qBMGs is given in~\cite{geiss2019best,korchmaros2023quasi}: A 2-qBMG  can be defined by the following three axioms where $\dG(V,E)$ stands for a digraph with vertex set $V$ and edge set $E$ which may have symmetric edges but no loops or parallel edges.    

\begin{itemize}
\item[(N1)] if $u$ and $v$ are two independent vertices then there exist no vertices $w,t$ such that $ut,vw,tw\in E$;
\item[(N2)] bi-transitive, i.e., if $uv,vw,wt\in E$ then $ut\in E$;
\item[(N3)] if $u$ and $v$ have common out-neighbor then either all out-neighbors of $u$ are also out-neighbors of $v$ or
all out-neighbors of $v$ are also out-neighbors of $u$.
\end{itemize}
Every 2-qBMG is a directed bipartite graph~\cite[Theorem 7.9]{korchmaros2023quasi}. Moreover, by \cite[Lemma 2.2]{korchmaros2023quasi}, a digraph $\overrightarrow{G}=\overrightarrow{G}(V,E)$ is a 2-qBMG if and only if it has properties (N1), (N2) and
\begin{itemize}
\item[(N3*)]
Let $u$ and $v$ be two vertices of the same color with a common out-neighbor such that there is no $w$ for which either $uw,wv\in E$ or $vw,wu\in E$. Then $u$ and $v$ have the same in-neighbors, whereas all out-neighbors of $u$ are also out-neighbors of $v$, or all out-neighbors of $v$ are also out-neighbors of $u$.
\end{itemize}
Although the configuration in (N1) is not related to mainstream graph theory and (N2) has been considered marginally yet, 2-qBMGs are found to have relevant properties for structural graph theory; in particular, their undirected underlying graph is free
from induced paths and cycles of size at least $6$~\cite{akITAT}. 

An important feature of 2-qBMGs is that they form a hereditary class of directed graphs with respect to their induced subgraphs~\cite{korchmaros2023quasi}. Theorem~\ref{pro02112024} shows that the family of 2-qBMGs is also hereditary for automorphisms in the sense that for any color-preserving automorphism group $\Gamma$ of a 2-qBMG, the $\Gamma$-quotient is also a 2-qBMG. Thus, if $\{1\}=\Gamma_0\lhd \Gamma_1\lhd \cdots \lhd \Gamma_k=\aut_I(\dG)$ is a subnormal series of the color-preserving automorphism group $\aut_I(\dG)$ of a 2-qBMG $\dG$, then the factors are also color-preserving automorphism groups of 2-qBMGs. Therefore, a sequence $\dG_0,\ldots,\dG_{k-1}$ of 2-qBMGs arises such that $\dG_i$ is the $\Gamma_i$-quotient of $\dG_{i+1}$ and the factor $\Gamma_{i+1}/\Gamma_i$ is a color-preserving automorphism group of $\dG_i$. If the subnormal series is normal, that is, $\Gamma_i\lhd \aut_I(\dG)$ for $i=1,\ldots,k-1$, then $\dG_i$ is a $\Gamma_i$ quotient of $\dG$ and $\aut_I(\dG)/\Gamma_i$ is a color-preserving automorphism group of $\dG_i$. 

It is also useful to look at the ``classical quotient digraph'' $\partial\dG$  whose vertices are the equivalence classes of vertices of $\dG$ w.r.t. the usual equivalence relation $\dot{\sim}$ in the vertex-set. 
Theorem~\ref{pro04112024}, valid for any bipartite graphs, shows that $\partial \dG$ is a $\Gamma$-quotient, and the structure of the group $\Gamma$ is completely described in Section \ref{sec4}. In particular, Theorem~\ref{pro04112024} states that $\Gamma\cong {\rm{Sym}}_{X_1}\times\cdots \times {\rm{Sym}}_{X_k}$ where $V=X_1\dot{\cup}\ldots\dot{\cup} X_k$ is the partition of the vertex-set $V$ whose parts are the equivalency classes w.r.t. $\dot{\sim}$. Conversely, let $X_1,X_2,\ldots,X_k$ be pairwise distinct sets. In Section \ref{sec4}, a general example is described, which provides 2-qBMGs whose $\dot{\sim}$ equivalence classes can be identified by $X_1,X_2,\ldots,X_k$. Therefore, there are plenty of 2-qBMGs with large color-preserving automorphism groups; see Theorem~\ref{th30102024}. Proposition~\ref{pro30102024}, which is the main ingredient in the proof of Theorem~\ref{th30102024}, also states that $\Gamma$ is a normal subgroup of $\aut_I(\dG)$. Thus, the factor group $\aut_I(\dG)/\Gamma$ is isomorphic to a subgroup of $\aut_I(\partial \dG)$. 

Therefore, to gain insight into the structure of the color-preserving automorphism groups of bipartite graphs, one has to look inside $\aut_I(\partial \dG)$. Since $\partial \dG$ is thin, i.e., no two vertices of $\partial \dG$ are equivalent w.r.t. $\dot{\sim}$, this leads to the study of color-preserving automorphism group of thin bipartite graphs. 

In this paper, we focus on thin 2-qBMGs. Theorem~\ref{pro20102024B} shows that thinness imposes severe restrictions on the structure of a 2-qBMG. In particular, if $\dG$ is a thin 2-qBMG, and $\aut_I(\dG)$ has only two orbits, namely the vertices of the same color, then the structure of $\dG=\dG(U\cup V,E)$ is completely determined as $\dG$ is either the union of pairwise disjoint directed starts or the union of pairwise disjoint symmetric edges. If $\aut_I(\dG)$ has more orbits, then this occurs for any two orbits $U_1\subset U$ and $V_1\subset V$ provided that the induced subgraph on $U_1\cup V_1$ has some edges.  
On the other hand, Theorem~\ref{pro20102024B} suggests a general example technique for thin 2-qBMGs that consists of gluing digraphs, which are the union of pairwise disjoint directed starts; see Constructions~\ref{ex2310204},~\ref{ex2510204} and~\ref{ex:general_case}. These thin 2-qBMGs contain no symmetric edge and have large automorphism groups. They show that for any two integers $m,s\ge 2$ there exist proper, thin 2-qBMGs on $2ms$ vertices whose color-preserving automorphism group contains a subgroup isomorphic to ${\rm{Sym}}_m$; see Theorem~\ref{the2410204}. So far, the only known examples of 2-qBMGs with large automorphism groups have been the complete bipartite digraphs $\overrightarrow{K}_{r,s}$ with color classes of size $r$ and $s$, respectively.  $\overrightarrow{K}_{r,s}$  is a 2-qBMG with only symmetric edges whose color-preserving automorphism group $\aut_I(\overrightarrow{K}_{r,s})$ is isomorphic to the direct product ${\rm{Sym}}_r\times {\rm{Sym}}_s$. 

In Section~\ref{oa}, we point out the particular role of symmetric edges in studying the automorphism groups of 2-qBMGs. Let $S$ consist of all vertices of a 2-qBMG shared by at least two symmetric edges. Then, every connected component of the subgraph induced by $S$ is a complete bipartite digraph~\cite[Theorem 4.16]{korchmaros2021structure}, and $\dG$, the induced digraph 2-qBMG on the complement of $S$, has the following property:  
\begin{itemize}
\item[($*$)] no two  symmetric edges of $\dG$  have a common endpoint.
\end{itemize}
Furthermore, any automorphism of a 2-qBMG leaves $S$ invariant; therefore, it is also an automorphism of $\dG$. From previous work on 2-qBMGs, it has emerged that 2-qBMGs $\dG$ satisfying $(*)$ have useful additional properties, such as the acyclicity of their orientations $O(\dG)$~\cite{korchmaros2021structure,korchmaros2023quasi}. Theorem \ref{thm:orientation2qBMG} shows that any orientation of a 2-qBMG satisfying $(*)$ is still a 2-qBMG. Furthermore, there is an orientation $O(\dG)$ consistent with automorphisms, i.e. $\aut(\dG)=\aut(O(\dG))$; see Theorem \ref{th06112024}. Therefore, the automorphisms of 2-qBMGs may be studied on 2-qBMGs free from symmetric edges.   

\section{Background} 
Our notation and terminology are standard; see \cite{babai1995automorphism,Biggs_1974,dobson2022symmetry}. 
\subsection{Graph theory}

In this paper, $\dG=(V,E)$ always denotes a digraph (directed graph) without loops and multiple edges, but digraphs without any edge are admitted. With the usual notation, $V=V(\dG)$ is its vertex-set, $E=E(\dG)$ is its edge-set where for $u,v\in V$, the edge with tail $u$ and head $v$ is denoted by $uv$ or $[u,v]$. 
For any vertex $v\in V$, $N^+(v)$ and $N^-(v)$ stand for the set of the out-neighbors and in-neighbors of $v$, respectively. A {\emph{symmetric edge}} is a pair of edges such that both $uv,vu \in E(\dG)$ hold. For a vertex $v$ of $\dG$, $v$ is a \emph{sink} if $N^+(v)=\emptyset$, and a \emph{source} if $N^-(v)=\emptyset$. In this paper, $G$ denotes an undirected graph, and $uv$ with $u,v\in V(G)$ is an undirected edge of $G$.

A digraph is \emph{oriented} if it does not have a symmetric edge. An oriented digraph has a \emph{topological vertex ordering} if its vertices can be labeled $u_1,u_2,\ldots,u_n$ such that for any edge $u_iu_j$ we have $i< j$. A vertex $u$ is \emph{minimal} for a topological vertex ordering of $\dG$ if $N^{-}(u)=\emptyset$; maximality for vertices is defined analogously. An \emph{orientation} $\overrightarrow{O}$ of $\dG$ is an oriented digraph obtained from $\dG$ by keeping the same vertex set but retaining exactly one edge from each symmetric edge. From the definition, $E(\dO)\subseteq E(\dG)$ whereas $V(\dG)=V(\dO)$.

A \emph{bipartite digraph} $\dG=\dG(V,E)$ is a digraph whose vertices are partitioned into two subsets (partition sets) such that the endpoints of every edge fall in different subsets. These two subsets $U$ and $W$ may be viewed as the color classes of a proper vertex coloring of two colors in the sense that no edge has its endpoints colored of the same color. A bipartite graph is \emph{balanced} if $|U|=|W|$.

Let $\dG=\dG(V,E)$ be any digraph. Take a vertex partition $V_1\dot{\cup} \cdots \dot{\cup} V_k$ of $V$.  The associated {\emph{quotient graph}}  is the digraph $\dG^*$ whose vertices are the parts $V_1,\ldots, V_k$ and $V_iV_j$ is an edge of $\dG^*$ if there exist $v_i\in V_i$ and $v_j\in V_j$ such that $v_iv_j\in E$.
For two digraphs $\dG$ and $\dG^*$, a \emph{homomorphism} is a surjective map $\varphi$ from $V(\dG)$ to $V(\dG^*)$ that preserves the edges. In this case, $\dG^*$ is the \emph{quotient graph} of $\dG$ w.r.t. the map $\varphi$.    

The equivalence relation $\dot{\sim}$ in a digraph $\dG$ is introduced by defining $x{\dot{\sim}}y$ whenever the vertices $x$ and $y$ have the same out-neighbors and in-neighbors. This gives rise to a particular quotient graph $\partial{\dG}$, the \emph{classical quotient digraph} of $\dG$, whose vertices are the equivalence classes w.r.t. the relation $\dot{\sim}$, and edges are as follows: Let $\bar{u}$ and $\bar{v}$ be two vertices of $\partial\dG=\partial\dG(\bar{V},\bar{E})$, then  $\bar{u}\bar{v}\in \bar{E}$ whenever $uv\in E$ for some (and hence for all) $u\in \bar{u}$ and $v\in\bar{v}$. If $\dG$ contains isolated vertices, then they form a unique equivalence class $C_0$, and hence $\partial \dG$ has a unique isolated vertex. If $\dG=\dG(V,E)$ is bipartite with color classes $U$ and $W$, then each equivalence class other than $C_0$ is contained in either $U$ or $W$. Therefore, $\partial\dG$ is also bipartite. An important property of bipartite digraphs $\dG$ is that $\partial{\dG}$ is thin, i.e. it contains no two distinct equivalent vertices, in other words $\partial({\partial{\dG}})=\partial{\dG}$.

The underlying undirected graph $G$ of $\dG$ has the same vertex set of $\dG$ whereas $xy\in E(G)$ if and only if either $xy\in E(\dG)$ or $yx\in E(\dG)$  for any two vertices $x,y\in V(G)$.

\subsection{Graphs and Groups}
An \emph{automorphism}  of a digraph $\dG=\dG(V,E)$ is a permutation $\pi$ on $V$ such that $xy\in E$ implies $\pi(x)\pi(y)\in E$ for any two vertices $x,y\in V$. The automorphisms of $\dG$ form a subgroup of the symmetric group $\rm{Sym}_V$ on $V$.  This group is the \emph{automorphism group} of $\dG$ and is denoted by $\aut(\dG)$. Each subgroup of $\aut(\dG)$ is called an automorphism group of $\dG$. For an undirected graph, $G=G(V,E)$, automorphisms and automorphism groups are defined analogously. 

If $\dG$ is a non-connected digraph and $\dG_1,\ldots, \dG_k$ are its connected components, then $\aut(\dG)$ has a subgroup isomorphic to $\aut(\dG_1)\times \cdots \times \aut(\dG_k)$.

If $\dG$ is connected and vertex-colored with two colors,  then an automorphism $h$ of $\dG$ is either  \emph{color-preserving}, or  \emph{color-switching} according as $h$ maps every vertex to another of the same, or opposite color, respectively; see \cite[Lemma 3.1]{MSM}. The color-preserving automorphisms of any bipartite graph $\dG$ form a subgroup $\aut_I(\dG)$ of $\aut(\dG)$. If $\dG$ is connected then either $\aut_I(\dG)=\aut(\dG)$ or $\aut_I(\dG)$ has index $2$ in $\dG$.  The hypothesis of being connected in the latter claim is necessary. 

For any automorphism group $H$ of $\dG$, the $H$-orbit $x^H$ of a vertex $x\in V(\dG)$ consists of all vertices $y\in V(\dG)$ such that there exists an automorphism $h\in H$ with $y=h(x)$. Since $y\in x^H$ implies $x\in y^H$, the $H$-orbits form a partition on $V(\dG)$. If two vertices $x,y$ are in the same $H$-orbits then $|N^+(x)|=|N^+(y)|$ and $|N^-(x)|=|N^-(y)|$, but the converse does not hold in general.

Let $\dG$ be a bipartite digraph with color classes $U$ and $W$. Take a partition for each color class, say $U=U_1\dot{\cup}\cdots \dot{\cup} U_r$ and $W=W_1\dot{\cup}\cdots \dot{\cup} W_s$, respectively. Let $\mathcal{C}(\dG)$ be the arising quotient graph. For a subgroup $\Gamma$ of $\aut_I(\dG)$, if all $U_i$ and $W_j$ are $\Gamma$-orbits then $\mathcal{C}(\dG)$ is denoted by $\dG/\Gamma$, and in this case $\dG$ is a  \emph{$\Gamma$-cover} of $\dG/\Gamma$ and  $\dG/\Gamma$ is the \emph{$\Gamma$-quotient} of $\dG$. If $\Gamma$ is a normal subgroup of $\aut_I(\dG)$ then $\dG/\Gamma$ has an automorphism group isomorphic to the factor group $\aut_I(\dG)/\Gamma$. If this is the case, then $\aut(\dG)/\Gamma$ is the \emph{inherited} automorphism group of $\dG/\Gamma$. 
For an intermediate subgroup $\Sigma$ between $\Gamma$ and $\aut_I(\dG)$, let $\dG/\Sigma$ be the associated $\Sigma$-cover. Then $\dG/\Sigma$ is also a $\Sigma$-cover of $\dG/\Gamma$. However, it should be stressed that the above claims do not hold for non-color-preserving automorphism groups chosen to play the role of $\Gamma$, as their quotients may not be bipartite graphs.

Every automorphism of $\dG$ is also an automorphism of its underlying undirected graph $G$, but the converse is false. In other words, $\aut(\dG)\le \aut(G)$ and equality does not generally hold.

\subsection{Two-quasi-best match graphs} 
It may happen that one or more of the defining axioms trivially hold in a 2-qBMG. If this is the case, then $\dG$ is (Ni)-\emph{trivial} for $1\le i \le 3$ if the hypotheses in axiom (Ni) are satisfied trivially. More precisely, $\dG$ is (N1)-trivial if there exists no quadruple  $(u,v,w,t)$ of vertices with $ut,vw,tw\in E(\dG)$; 
$\dG$ is (N2)-trivial if there exists no quadruple of vertices with $uv,vw,wt\in E(\dG)$, and $\dG$ is (N3)-trivial if there exist no two distinct vertices with a common out-neighbor. If $\dG$ is (N3)-trivial, then it is also 
(N3*)-\emph{trivial}. If $\dG$ is non-trivial for (N2), then it is not trivial for (N1) and (N3), as well, and $\dG$ is a \emph{proper} 2-qBMG. 
In the extreme case, i.e. if all three defining properties hold trivially, $\dG$ is (N)-\emph{trivial}.

For four vertices $x_1,x_2,x_3,y$ of $\dG$, we say that $[x_1,x_2,x_3,y]$ is an (N1)-\emph{configuration} if  $x_1x_2,x_2x_3,yx_3 \in E(\dG)$ implies either $x_1y\in E(\dG)$ or $yx_1\in E(\dG)$; in other words when condition (N1) holds for $u=x_1,t=x_2,w=x_3,v=y$ and therefore $x_1$ and $y$ are not independent.

Orientations of a 2-qBMG which is either thin, or enjoys property $(*)$, are acyclic and hence have a topological ordering. An approach to sink-free 2-qBMGs (also called 2-BMGs), which is based on topological orderings of their orientations, is worked out in~\cite{korchmaros2021structure}, see also~\cite{korchmaros2023quasi}.

\section{\texorpdfstring{$\Gamma$}-quotients of 2-qBMGs}
\label{seclarge}
The relevance of $\Gamma$-quotients in studying 2-qBMGs depends on the following hereditary property.
\begin{theorem}
\label{pro02112024} 
If $\dG$ is a 2-qBMG and $\Gamma$ is a subgroup of $\aut_I(\dG)$ then the $\Gamma$-quotient $\dG/\Gamma$ is also a 2-qBMG. 
\end{theorem}
\begin{proof} 
Let $E$ denote the edge-set of $\dG$, and $\mathcal{E}$ the edge-set of $\dG/\Gamma$. 
Take a (N1)-configura-\\tion $[U_1,W_1,U_2,W_2]$ in $\dG/\Gamma$.  By definition,  there exist $u_1\in U_1,w_1,w_1^*\in W_1, u_2,u_2^*\in U_2,w_2\in W_2$ such that 
$u_1w_1,w_1^*u_2,w_2u_2^*\in E$. Since $w_1,w_1^*\in W_1$, there is an automorphism $\gamma\in \Gamma$ which takes $w_1$ to $w_1^*$.  Then $u_2^+=\gamma^{-1}(u_2)\in U_2$, as $U_2$ is a $\Gamma$-orbit. Thus  $w_1u_2^+$ $=\gamma^{-1}(w_1^*)\gamma^{-1}(u_2)\in E$, as $\gamma\in \aut(\dG)$. Similarly, an automorphism $\gamma^*\in \Gamma$ takes $u_2$ to $u_2^+$. For $w_2^*=\gamma^*(w_2)$, we have $w_2^*u_2^+\in E$. 
Therefore $[u_1, w_1,u_2^+,w_2^*]$ is a (N1)-configuration in $\dG$, yielding either $u_1w_2^*$, or $w_2^*u_1\in E$. Hence, either $U_1W_2\in \mathcal{E}$, or $W_2U_1\in \mathcal{E}$. This shows that (N1) is satisfied by $\dG/\Gamma$. Analogous arguments can show that (N2) holds in $\dG/\Gamma$. 

Finally, take any two vertices of  $\dG/\Gamma$ that have a common out-neighbor: they are $u_1\in U_1,\,u_2\in U_2$ and there exist $w_1,w_1^*\in W_1$ such that $u_1w_1,u_2w_1^*\in E$. There is an automorphism $\gamma\in \Gamma$ taking $w^*$ to $w$. Let $u_2^*=\gamma(u_2)$. Then $w$ is a common neighbor of $u_1$ and $u_2^*$. Thus (N3) applies in $\dG$  yielding either $N^+(u_1)\subseteq N^+(u_2^*)$ or $N^+(u_2^*)\subseteq N^+(u_1)$.
In the former case, take $W_1\in N^+(U_1)$. Hence $u_1^*w_1^*\in E$ for some $u_1^*\in U_1$ and $w_1^*\in W_1$. Let $\gamma\in \aut(\dG)$ such that $\gamma(u_1^*)=u_1$. Then $u_1\gamma(w_1^*)=\gamma(u_1^*)\gamma(w_1^*)=\gamma(u_1^*w_1^*)\in E$. Hence $\gamma(w_1^*)\in N^+(u_1)\subseteq N^+(u_2)$, yielding $W_1\in N^+(U_2)$ and $N^+(U_1)\subseteq N^+(U_2)$. Analogous arguments  show that $N^+(U_2)\subseteq N^+(U_1)$ follows in the latter case.  Therefore, (N3) is satisfied by $\dG/\Gamma$.
\end{proof}

If we weaken the hypothesis $\Gamma$-quotient to that of any (not necessarily $\Gamma$-) quotient, Theorem~\ref{pro02112024} may not hold any longer.   
Indeed, the quotient bipartite digraph of the 2-qBMG in Figure~\ref{figura10} is not bi-transitive. 

Furthermore, if the $\Gamma$-quotient of a bipartite digraph $\dG$ is a 2-qBMG, it may be that $\dG$ itself is not a 2-qBMG, as shown in Figure~\ref{figura11}. Indeed, (N2) does not hold for $\dG$.  

Finally, if $\dG$ is a 2-qBMG and $\Gamma$ is chosen to be the full $\aut_I(\dG)$, then the associated $\Gamma$-quotient $\dG/\aut_I(\dG)$ has no non-trivial automorphism, in general. Nevertheless, there are counterexamples such as the 2-qBMG in Figure~\ref{figura9}. 
\begin{figure}[ht!]
\centering
\scalebox{0.45}
{\includegraphics{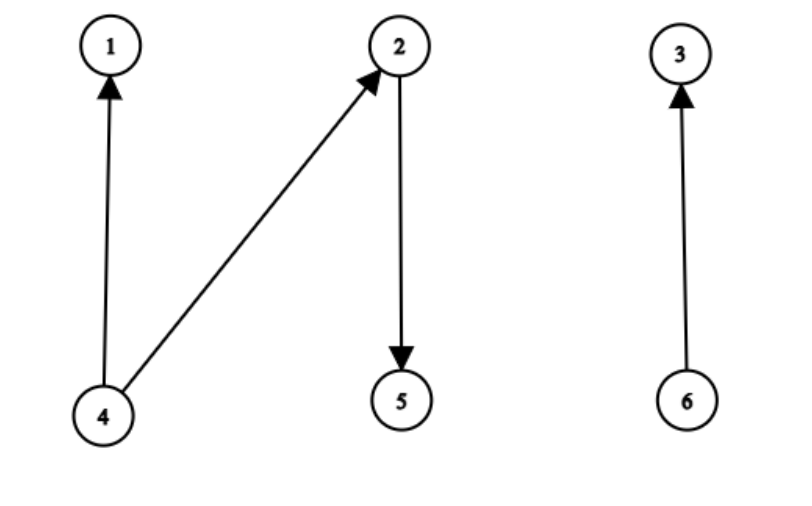}}
\caption{Non-bi-transitive quotient of 2-qBMG.  The graph $\dG$ is a 2-qBMG with color classes $U=\{1,2,3\}$ and $W=\{4,5,6\}$, and covering $U_1=\{1\},U_2=\{2\},U_3=\{3\}$, $W_1=\{4\}, W_2=\{5,6\}$. The quotient bipartite digraph $\dG_0$ has color classes $\mathcal{U}=\{U_1,U_2,U_3\}$ and $\mathcal{W}=\{W_1,W_2\}$ and edge set $\mathcal{E}=\{W_1U_1,W_1U_2,U_2W_2,W_2U_3\}$. $\dG_0$ does not satisfy (N2) as $W_1U_2,U_2W_2,W_2U_3
\in \mathcal{E}$ but $W_1U_3\not\in \mathcal{E}$.}
\label{figura10}
\end{figure}
\begin{figure}[ht!]
\centering
\scalebox{0.40}
{\includegraphics{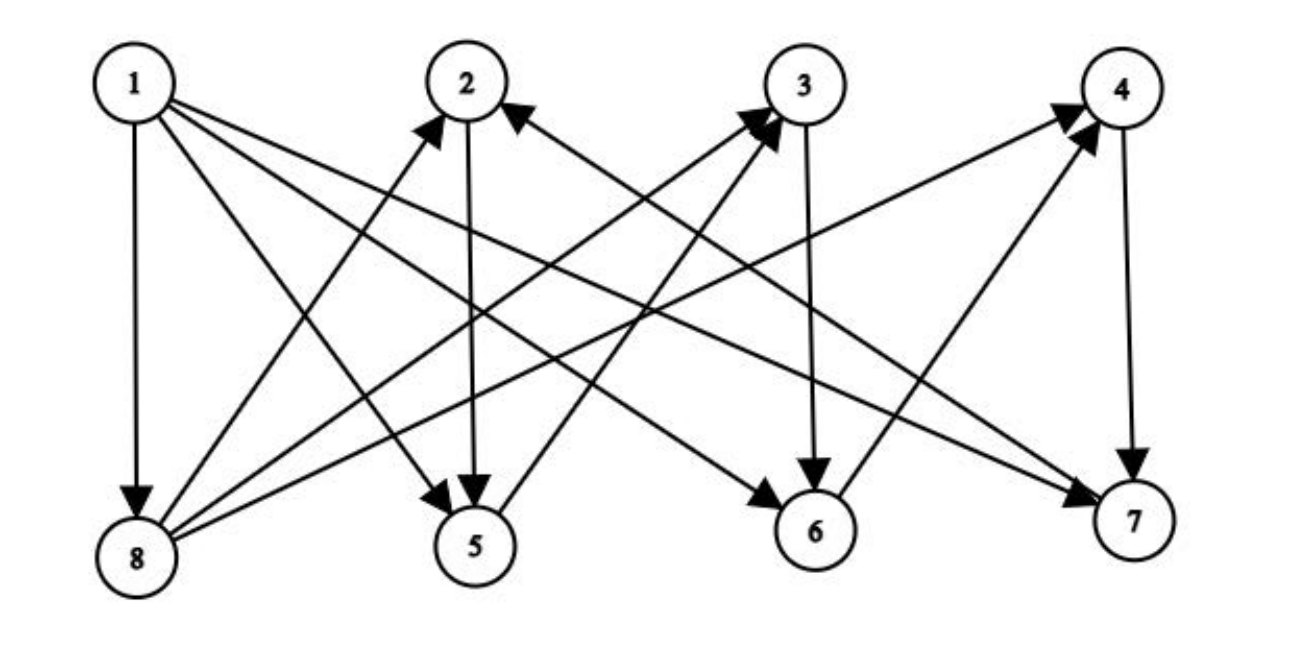}}
\caption{Non-bi-transitive digraph of 2-qBMG $\Gamma$-quotient. $\{1,2,3,4\}$ and $\{5,6,7,8\}$ are the color classes of $\dG$. $[2,6]\notin E$, (N2) does not hold in $\dG$.  $\aut_I(\dG)$ is a group of order $3$ and it has four orbits: $o_1=\{1\},o_2=\{8\},o_3=\{2,3,4\},o_4=\{5,6,7\}$. For $\Gamma=\aut_I(\dG)$, the $\Gamma$-covering is the 2-qBMG with color classes $\{o_1,o_3\},\,\{o_2,o_4\}$ and edge-set $\{[o_1,o_2],[o_2,o_3],[o_3,o_4],[o_1,o_4]\}$.}
\label{figura11}
\end{figure}
\begin{figure}[ht!]
\centering
\scalebox{0.55}
{\includegraphics{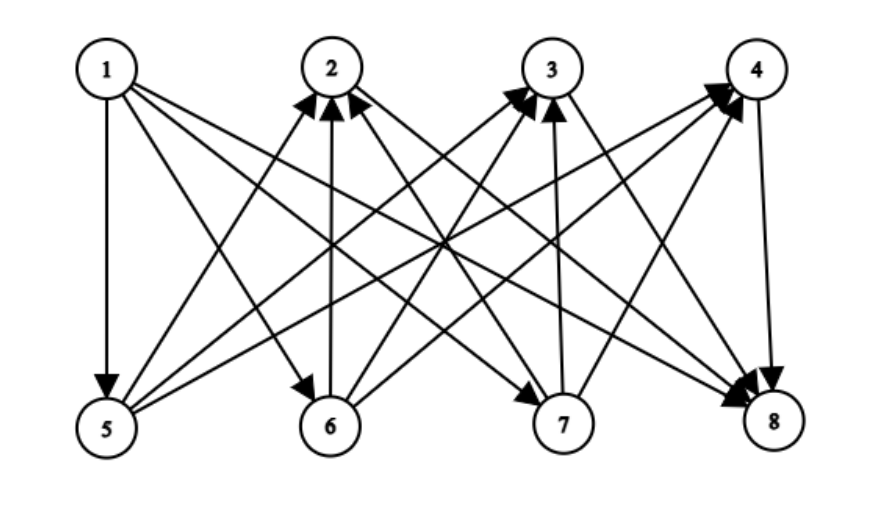}}
\caption{Non-trivial automorphisms for $\Gamma$-quotient. $\{1,2,3,4\}$ and $\{5,6,7,8\}$ are the color classes of $\dG$. $|\aut_I(\dG)|=36$ and $\aut_I(\dG)$ has four orbits: $o_1=\{1\}$, $o_2=\{8\}$, $o_3=\{2,3,4\}$ and $o_4=\{5,6,7\}$. The quotient $\dG/\aut_I(\dG)$ is a 2-qBMG on $\{o_1,o_2,o_3,o_4\}$ with color classes $\{o_1,o_3\}$ and $\{o_2,o_4\}$, and edge-set  $\{o_1o_2,o_1o_4,o_3o_2,o_4o_3\}$. $\aut_I(\dG/\aut_I(\dG))$ contains the involutory permutations $(o_1o_3),(o_2,o_4)$.}
\label{figura9}
\end{figure}

\section{Automorphism groups of digraphs containing equivalent vertices}
\label{sec4}
We show a useful result on automorphism groups of bipartite digraphs that contain equivalent vertices. The first is a straightforward consequence of the definition of equivalency $\dot{\sim}$ and the notion of an automorphism.  

\begin{lemma}
\label{lem30102024} For a vertex $x$ of a bipartite digraph $\dG$ on $n$ vertices, let $X$ denote the set of all vertices of $\dG$ which are equivalent to $x$. Let $H$ be the subgroup of ${\rm{Sym}}_n$ such that
\begin{itemize}
\item[(i)] $H$ leaves $X$ invariant;
\item[(ii)] $H$ acts on $X$ as the symmetric group on $|X|$ labels;  
\item[(iii)] $H$ fixes every vertex off $X$.
\end{itemize}
Then $H$ is an automorphism group of $\dG$.     
\end{lemma}
\begin{proof}  Let $U$ and $W$ be the color classes of $\dG$ where $x\in U$. Since $x$ is not an isolated vertex, any vertex in $X$
is in $U$. Take a vertex $w\in W$ in $N^+(x)$. Since $w\not\in X$, every automorphism $h\in H$ fixes $w$. Therefore the edge $xw$ of $\dG$ is mapped by  $h$ to $h(x)w$. Since $x\dot{\sim} h(x)$, we have that $h(x)w$ is also an edge of $\dG$. The same argument shows that if $w\in N^-(x)$, then $wh(x)$ is an edge of $\dG$. Since $h$ fixes each vertex off $X$, the claim follows.   
\end{proof}

For a bipartite digraph $\dG=\dG(V,E)$, let $V=X_1\cup\ldots\cup X_k$ be the partition of $V$ whose members are the equivalency classes w.r.t. the relation $\dot{\sim}$. For each $X_i$ let $H_i$ be the automorphism group of $\dG$ as defined in Lemma~\ref{lem30102024}.    
\begin{proposition}
\label{pro30102024} The subgroup $\Gamma=\langle H_1,\ldots,H_K\rangle$ of $\aut(\dG)$ which is generated by $H_1,\ldots,H_k$ is the direct product $H_1\times\cdots \times H_k$ and $\Gamma$ is a normal subgroup of $\aut(\dG)$. 
\end{proposition}
\begin{proof} From the definition, $h_ih_j=h_jh_i$ for $h_i\in H_i,\,h_j\in H_j$ whence $H_iH_j=H_jH_i$. Furthermore, $H_i\cap H_j=\{1\}$ for $1\le i,j \le k$. Therefore $\langle H_1,\ldots,H_K\rangle=H_1\times \cdots \times H_k$. Take $g\in \aut(\dG)$ and $H_i$ for $1\le i \le k$. We have to show that $g\circ H_i\circ g^{-1}\in \{H_1,\ldots,H_k\}$ for any $1\le j \le k$. Since $g$ takes equivalent vertices to equivalent vertices, if there exists $\xi_j\in H_j$ such that $g(\xi_j)\in H_i$ then for every $x_j\in X_j$ we have $g(x_j)\in X_i$. From this $|H_j|=|H_i|$, as $g$ is a bijection on $V$.
Moreover, for $h_i \in H_i$, since $x'_i\in X_i$, $$g\circ h_i\circ g^{-1}(x_j)=(g\circ h_i)(x_i)=g(h_i(x_i))=g(x'_i)\in X_j.$$  This shows that $g\circ H_i \circ g^{-1}$ leaves $X_j$ invariant. Furthermore, for $x_m\in X_m$ with $m\ne j$, we have $x_t=g^{-1}(x_m)\notin X_i$; in particular $h_i(x_t)=x_t$. 
Therefore, $g\circ h_i\circ g^{-1}(x_m)=(g\circ h_i)(x_t)=g(h_i(x_t))=g(x_t)=x_m$. This shows that $g\circ H_i \circ g^{-1}$ fixes any vertex outside $X_j.$ Then, by definition, $g\circ H_i \circ g^{-1}\le H_j.$  Since $|H_i|=|H_j|$, this yields $g\circ H_i \circ g^{-1}= H_j$. Finally, let $g\in\aut_I(\dG)$ such that $g$ preserves $X_i$ for $i=1,2,\ldots, k$. Then $g$ induces a permutation on $X_i$, and hence there exists $h_i\in H_i$ such that $g(x_i)=h_i(x_i)$ for every $x_i\in X_i$. With these $h_i$ for $i=1,2,\ldots,k$, we have $g=h_1h_2\cdots h_k$. 
\end{proof}
The following result is obtained since the orbits of $\Gamma$ are exactly the equivalence classes w.r.t. the relation $\dot{\sim}$. 
\begin{theorem}
\label{pro04112024} If $\dG$ is a bipartite digraph, then its quotient digraph $\partial\dG$ is the $\Gamma$-quotient of $\dG$ with $\Gamma=H_1\times\cdots \times H_k$.
\end{theorem}
The converse of Proposition \ref{pro30102024} also holds.  
\begin{proposition}
\label{pro05112924} Let $\dG$ be a bipartite digraph with color classes $U,W$ and partitions $U=X_1\cup\cdots\cup X_r$,  
$W=X_{r+1}\cup\cdots\cup X_k$. For $1\le i \le k$ suppose that $\aut_I(\dG)$ has a subgroup $H_i\cong {\rm{Sym}}_{X_i}$ that satisfies (i),(ii), (iii) of Lemma \ref{lem30102024} for $X=X_i$ with $1\le i \le k$. Then each  $X_1,\ldots,X_k$ consists of pairwise equivalent vertices and $\Gamma=\langle H_1,\ldots,H_k\rangle=H_1\times \cdots \times H_k$. 
\end{proposition}
\begin{proof} For $1\le i \le k$, take $x,y\in X_i$. Choose an automorphism $h_i\in H_i$ that takes $x$ to $y$. Let $v\in V$ such that $v\in N^+(x)$. Then $v\not\in X_i$. Therefore, $h_i$ takes the ordered pair $(x,v)$ to $(h(x),h(v))=(y,v)$. Since $xv\in E(\dG)$ and $h$ preserves the edges, this yields that $yv\in E(\dG)$. Replacing $N^+(x)$ with $N^-(x)$ in the above argument shows that $vx\in E(\dG)$ implies $vy\in E(\dG)$.   
\end{proof}
\begin{remark} In Proposition \ref{pro05112924}, if we weaken  Hypothesis 
(ii) to requiring only the transitivity of $H_i$ on $X_i$, then its proof still works.
\end{remark}

Given any bipartite digraph, Proposition \ref{pro30102024} gives an idea of how to obtain from it some bipartite digraphs with large automorphism groups. 
Start with a bipartite digraph $\overrightarrow{\Gamma}=\overrightarrow{\Gamma}(S,F)$ with vertex-set $S$ and edge-set $F$. For a vertex $s\in S$, consider the directed bipartite subgraph $\overrightarrow{\Gamma}_s$ of $\overrightarrow{\Gamma}$ which is induced by $s$ together with all its out-neighbors and in-neighbors. Add a vertex $s'$ to $V$ together with all edges $s'u$ and $vs'$ where $u\in N^+(s)$ and $v\in N^-(s)$. Then, the resulting bipartite digraph contains equivalent vertices, namely $s$ and $s'$. This procedure is called \emph{blow-up} by analogy with algebraic geometry.  Blow-up can be repeated several times using the same vertex $s$, building up an equivalence class of vertices with any size. This procedure may be repeated for any other vertex so that $k$ equivalence classes $X_1,\ldots,X_k$ arise. For the bipartite digraphs $\dG$ constructed in this manner, Proposition \ref{pro30102024} shows that if $H_i$ is the symmetric group on $X_i$, then $\aut(\dG)$ contains a subgroup isomorphic to the direct product $H_1\times \cdots \times H_k$. 
\begin{example}
\label{ex010125}
{\em{Let $\dG$ be the 2-qBMG in 
Figure~\ref{figura100}(left). The blow-up at vertex $1$ applies to $\dG$ by adding vertex $6$. The arising digraph $\dG[1]$ is the 2-qBMG in Figure~\ref{figura100}(right). 
\begin{figure}[ht]
\centering
\scalebox{0.30}
{\includegraphics{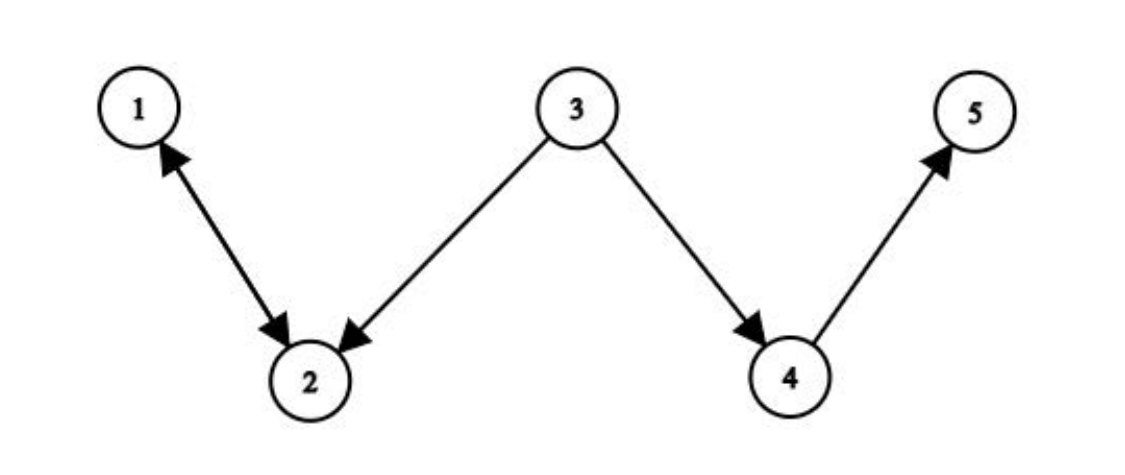}}
\scalebox{0.30}
{\includegraphics{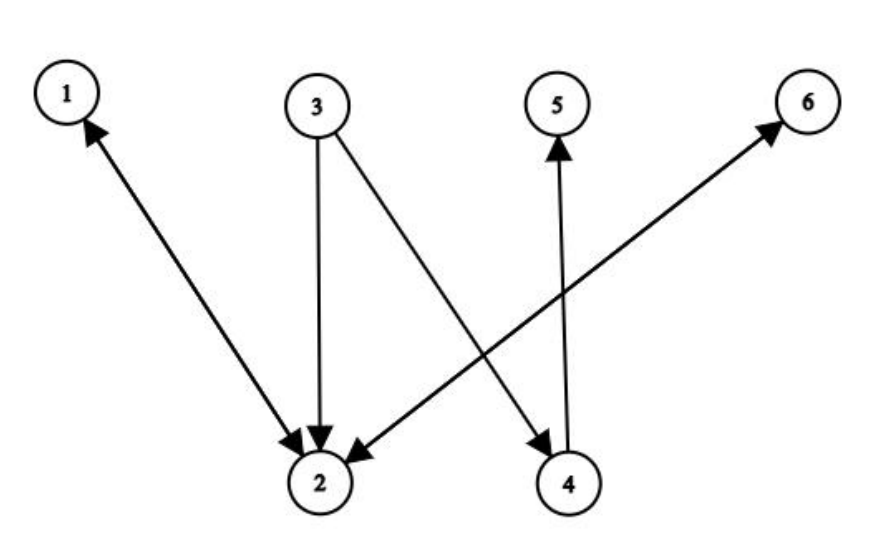}}
\caption{First blow-up. $\dG$ is the 2-qBMG  with color classes $\{1,3,5\},\,\{2,4\}$, and edge-set $E:=\{[1,2],[2,1],[3,2],[3,4],[4,5]\}$ (left). $\dG[1]$ is the 2-qBMG with color classes $\{1,3,5,6\}$, 
$\{2,4\}$, and edge-set $\{[1,2],[2,1],[3,2],[3,4],[4,5],[6,2],[2,6]\}$, obtained by adding vertex $7$ to $\dG[1]$. $\dG[1]$ is the blow-up at vertex $1$ applied to $\dG$ (right).}
\label{figura100}
\end{figure}
The blow-up at vertex $2$ applies to $\dG[1]$ by adding vertex $7$. The arising digraph $\dG[1][2]$ is the 2-qBMG in Figure~\ref{figura101}(left). 
\begin{figure}[ht]
\scalebox{0.30}
{\includegraphics{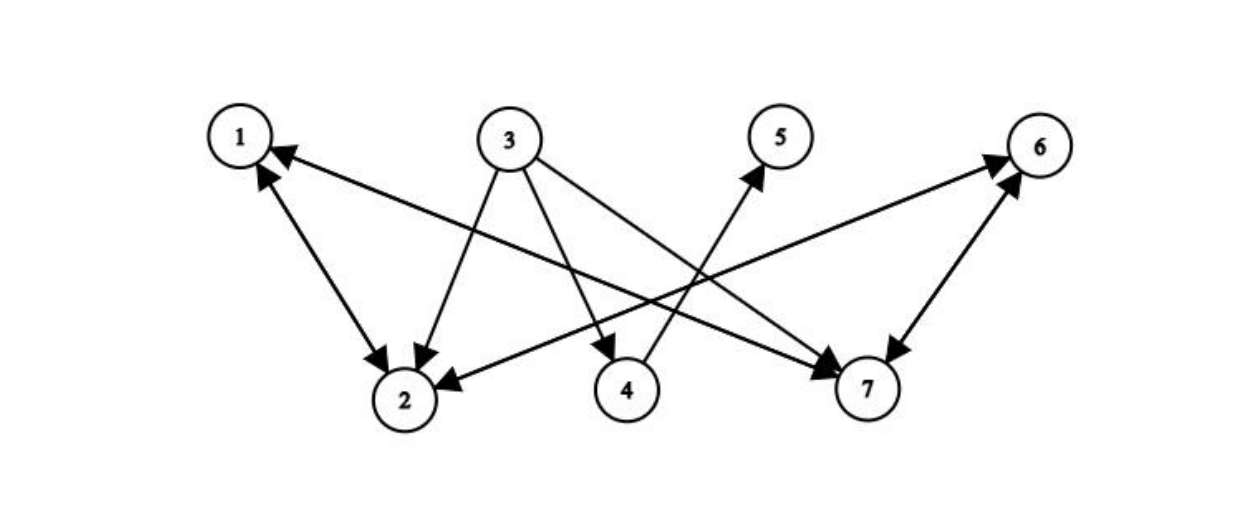}}
\scalebox{0.30}
{\includegraphics{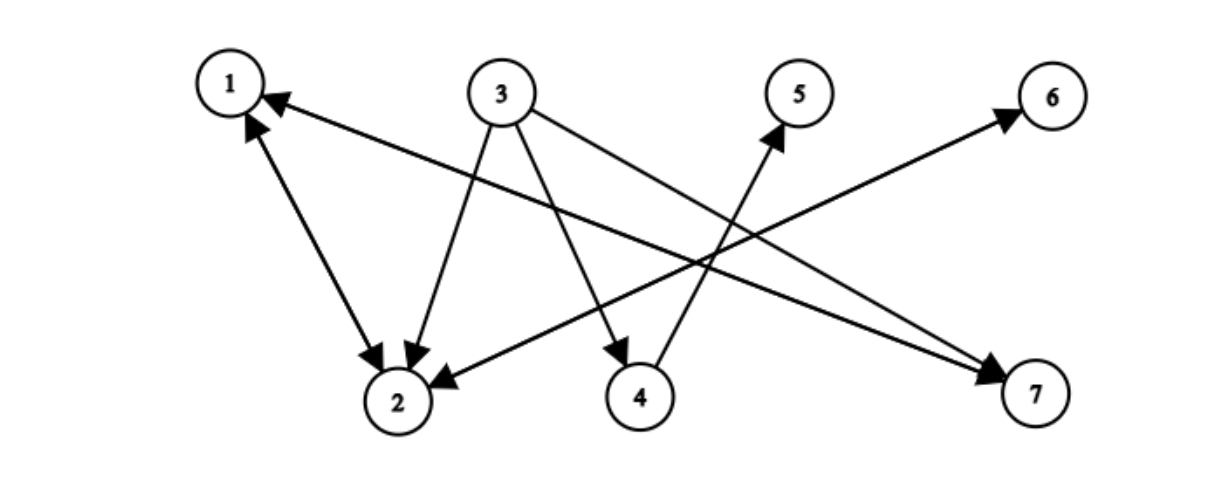}}
\caption{Second blow-up. $\dG[1][2]$ is the 2-qBMG with color classes $\{1,3,5,6\}$, $\{2,4,7\}$ and edge-set $\{[1,2],[2,1],[3,2],[3,4],[4,5],[6,2],[2,6]\}$, obtained by adding vertex $7$ to $\dG[1]$. $\dG[1][2]$ is the blow-up at vertex $2$ applied to $\dG[1]$ (left). $\dG[1,2]$ is the digraph, obtained by adding $6$ and $7$ simultaneously, with color classes $\{1,3,5,6\}$, $\{2,4,7\}$, and edge-set $E\cup \{[6,2],[2,6],[7,1],[1,7],[3,7]\}$ is not a 2-qBMG (right).}
\label{figura101}
\end{figure}
Permutations $(1,6)$ and $(2,7)$ are color-preserving automorphisms of $\dG[1][2]$. 
However, blowing up $\dG$ at vertices $1$ and $2$ simultaneously does not produce a 2-qBMG; see; see Figure~\ref{figura101}(right).    Indeed, if we add $6$ and $7$ to the vertex set of $\dG$ and extend the edge-set by adding $[6,2],[2,6],[7,1],[1,7],[3,7]$ so that $1\dot{\sim}6$ and $2\dot{\sim}7$.}}
\end{example}

Now if we start with any 2-qBMG, the resulting digraph $\dG$ is a 2-qBMG, as well. Therefore, the following result is obtained.
\begin{theorem}
\label{th30102024} Let $X_1,X_2,\ldots,X_k$ be pairwise distinct sets. Then there exists a 2-qBMG whose color invariant automorphism group $\aut_I(\dG)$ contains a normal subgroup, which is the direct product $H_1\times\cdots\times H_k$ where $H_i$ is isomorphic to the symmetric group on $X_i$.   
\end{theorem}
\section{Automorphism groups of thin 2-qBMGs}
Proposition \ref{pro30102024} shows that the structure of the automorphism group of a digraph is determined by that of its quotient digraph. This motivates the study of automorphism groups of thin digraphs, particularly in the context of the present paper, the study of the automorphism groups of thin 2-qBMGs.

In this section,  $\dG=\dG(V,E)$ always denotes a 2-qBMG. Therefore, $\dG$ has properties (N1),(N2),(N3) and (N3*). In particular, $\dG$ is a bipartite graph with color classes $U$ and $W$ where $V=U\cup W$. 
A key property is stated in the following claim.
\begin{proposition}
\label{pro20102024}
If $x,y$ are from the same $\aut_I(\dG)$-orbit and have a common out-neighbor, then they are equivalent.
\end{proposition}
\begin{proof}  Take $x,y\in V$ be with the following property. There exist $h\in \aut(\dG)$ such that $y=h(x)$ and $w\in V$ such that $w\in N^+(x)\cap N^+(y)$. We have to show that $x\dot{\sim} y$. From the hypothesis, $x$ and $y$ are in the same color class, say $U$. Three cases are investigated separately.

Case (I). {\em{Some out-neighbor of $x$ is an in-neighbor of $y$.}} In this case, there exists $w\in W$ for which $w\in N^+(x)$ and $w\in N^-(y)$ so that $xw,wy\in E$. We show that $N^-(x)\subseteq N^-(y)$. Take $z\in W$ with $z\in N^-(x)$. Then $zx\in E$, and (N2) applies to $zxwy$ whence $zy\in E$. Therefore, $z\in N^-(y)$. Thus $N^-(x)\subseteq N^-(y)$.  Furthermore, $|N^-(x)|=N^-(y)|$ by $y=h(x)$ whence $N^-(x)=N^-(y)$ follows. To show $N^+(x)=N^(y)$ it is enough to observe that (N3) applies to the pair $\{x,y\}$. Since $|N^+(x)|=N^+(y)|$ by $y=h(x)$, neither $N^+(x)\subsetneq N^+(y)$ nor $N^+(y)\subsetneq N^+(x)$ are possible. Therefore, $N^+(x)=N^+(y)$. Thus $x\dot{\sim} y$, a contradiction. 

Case (II). {\em{Some out-neighbor of $y$ is an in-neighbor of $x$.}} The proof of Case (I) holds true prior to swapping $x$ and $y$.

Case (III). {\em{The hypotheses in (N3*) are satisfied.}} In this case, (N3*) implies that $N^-(x)=N^-(y)$ and that either $N^+(x)\subseteq N^+(y)$ or $N^+(y)\subseteq N^+(x)$. Since $|N^+(x)|=|N^+(y)|$, the latter case only occurs when $N^+(x)=N^+(y)$. Therefore $x\dot{\sim} y$.
\end{proof}
\subsection{Orbits of \texorpdfstring{$\aut(\dG)$} on vertices}
\begin{proposition}
\label{pro20102024A} Assume that $\dG$ is a thin 2-qBMG. If an automorphism $g$ of $\dG$ fixes a vertex $v\in V$, then $g$ fixes each vertex in the in-neighborhood of $v$.
\end{proposition}
\begin{proof} Take $g\in\aut(\dG)$ and $v\in V$ such that $g(v)=v$. On the contrary, assume that $x\in N^-(v)$ exists, such as $g(x)\ne x$. Let $y=g(x)$. Then $y$ belongs to the orbit $x^H$ where $H$ is the subgroup generated by $g$. Furthermore, $v$ is an out-neighbor of $x$ and $y$. Therefore, Proposition~\ref{pro20102024} applies whence $x\dot{\sim} y$ follows, a contradiction.
\end{proof}
\begin{theorem}
\label{pro20102024B} Assume that $\dG$ is a thin 2-qBMG. For any two subsets $U_1\subseteq U$ and $W_1\subseteq W$ which are $\aut(\dG)$-orbits, let $\overrightarrow{G}_1=\dG_1(V_1,E_1)$ be the directed subgraph of $\dG$ induced on $V_1=U_1\cup W_1$. If there exist $x_1\in U_1, y_1\in W_1$ such that $x_1y_1\in E$ then one of the following cases occurs:
\begin{itemize}
\item[(i)]  $\dG_1$ is the union of pairwise disjoint stars where $U_1$ comprises sources while $W_1$ does sinks, and $|N^+(x_1)||U_1|=|W_1|$; in particular if $|U_1|=|W_1|$ then $E_1$ is the union of pairwise disjoint edges. 
\item[(ii)] $E_1$ is the union of pairwise disjoint symmetric edges.  
\end{itemize}
\end{theorem}
\begin{proof} Take $x_1,y_1\in V$ such that $x_1y_1\in E$. Assume that there exist $x_2\in U_1$ and $y_2\in W_1$ such that $y_2x_2\in E$. Since $W_1$ is a $\aut_I(\dG)$-orbit, we find some $\pi\in \aut(\dG)$ such that $\pi(y_2)=y_1$. Let $x_3=\pi(x_2)$. Then $x_1y_1, y_1x_3\in E$. Also, $x_3\in U_1$ as $U_1$ is an $\aut_I(\dG)$-orbit. Two cases occur according as $x_1 \ne x_3$ or $x_1= x_3.$ 

In the former case, the arguments in Case (I) of the proof of Proposition \ref{pro20102024} can be used for $x=x_1, w=y_1, y=x_3$ to prove that $x_1{\dot{\sim}} x_3$. Since $\dG$ is thin, this is a contradiction. Thus,  as $U_1$ and $W_1$ are $\dG$-orbits, each vertex in $U_1$ is a source whereas each vertex in $W_1$ is a sink of $\dG_1$. We show the second claim in (i). From Proposition \ref{pro20102024}, $N^+(x_1)\cap N^+(u_1)=\emptyset$ for any two distinct $x_1,u_1\in U_1$. Since $U_1$ is a $\aut_I(\dG)$-orbit, $|N^+(x_1)|=|N^+(u_1)|$ holds for any $x_1,u_1\in U_1$. Therefore, $|N^+(x_1)||U_1|=|W_1|$ for $x_1\in U_1$. If $|U_1|=|W_1|$ then $|N^+(u_1)|=|N^+(x_1)|=1$ for every $u_1\in U_1$. Since $|N^-(w)|=1$ also holds, $\dG_1$ is regularly $N$-trivial.  

In the latter case, $x_1=x_3$, and hence $x_1y_1$ is a symmetric edge in $E_1$. Since $U_1$ is an $\aut_I(\dG)$-orbit, each vertex in $U_1$ is an endpoint of a symmetric edge. Since the same is valid for $W_1$, the first claim in (ii) holds. From Proposition~\ref{pro20102024}, there exists no $x_2$ in $U_1$ other than $x_1$ such that $x_2y_1\in E$ holds. Similarly, no $y_2$ in $W_1$ other than $y_1$ satisfies $x_1y_2\in E_1$. Therefore, the second claim in (ii) also holds.       
\end{proof}
A corollary of Proposition \ref{pro20102024A} is stated in the following proposition. 
\begin{proposition}
\label{pro20102024C} Assume that $\dG=\dG(U\cup W,E)$ is a thin 2-qBMG. If $\dG$ is balanced and both color classes $U$ and $W$ of $\dG$ are $\aut_I(\dG)$-orbits, then $E$ is either the union of pairwise disjoint edges, or the union of pairwise disjoint symmetric edges, and  $\aut_I(\dG)\cong {\rm{Sym}}_{\it{n}}$ with $n=|U|=|W|$.   
\end{proposition}
\subsubsection{Examples for Theorem ~\ref{pro20102024B}}
We construct three thin 2-qBMGs using the ideas explained in Theorem~\ref{pro20102024B}.
\begin{example} 
\label{ex2310204}
Take four pairwise disjoint sets $U_1,U_2,W_1,W_2$ of the same size $m$ consisting of positive integers, namely 
$$
\begin{cases} 
U_1=\{1,2,\ldots,m\},\,\, U_2=\{m+1,m+2\ldots,2m\}\\ 
W_1=\{2m+1,2m+2\ldots,3m\},\,\, W_2=\{3m+1,3m+2\ldots,4m\}.
\end{cases}
$$
Moreover, choose three invertible functions, say $\alpha,\beta, \gamma$, where $$\alpha: U_1\mapsto W_1,\,\beta:=W_1 \mapsto U_2,\, \gamma:=U_2\mapsto W_2.$$
Then $\delta=\gamma\circ\beta\circ\alpha$ is an invertible function $\delta:U_1\mapsto W_2$. Consider the bipartite digraph $\dG$ with color classes $U=U_1\cup U_2$ and $W=W_1\cup W_2$ where the edge set $E$ is defined as follows: for $u_1\in U_1$ and $w_1\in W_1$, $E$ contains $u_1w_1$ if and only if $w_1=\alpha(u_1)$; for $w_1\in W_1,u_2\in U_2$, $E$ contains $w_1u_2$ if and only if $u_2=\beta(w_1)$; for $u_2\in U_1,w_2\in W_2$, $E$ contains $u_2w_2$ if and only if $w_2=\gamma(u_2)$; and finally for $u_1\in U_1, w_2\in W_2$ if and only if $w_2=\delta(u_1)$. No more edge is added to $E$. By example, $\dG(V,E)$ with $V=U\cup W$ is thin and has both properties (N1) and (N2). Furthermore, property (N3) holds as $|N^+(v)|\leq 1$ for all $v\in V$. Hence, $\dG$ is a thin proper 2-qBMG. 
Moreover, the sources of $\dG$ are the vertices in $U_1$, whereas the sinks of $\dG$ are those in $V_2$. Figure~\ref{figura7} illustrates the case for $m=4$,
where  
$$\alpha=\left[
  \begin{array}{cccc}
    1 & 2 & 3 & 4\\
    10 & 9 & 12 & 11\\
     \end{array},
\right] \quad
\beta=\left[
  \begin{array}{cccc}
    9 & 10 & 11 & 12\\
    8 & 6  & 7 & 5\\
     \end{array},
\right]
\quad
\gamma=\left[
  \begin{array}{cccc}
    5 & 6 & 7 & 8\\
    14 & 13 & 15 & 16 \\
     \end{array}
\right]
.
$$
\begin{figure}[ht!]
\centering
\scalebox{0.55}
{\includegraphics{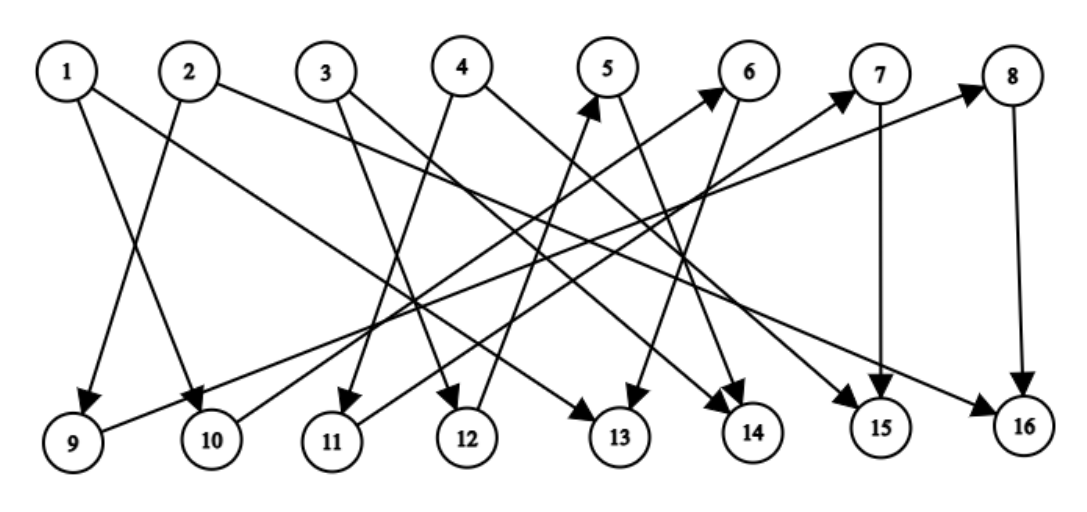}}
\caption{Case $m=4$. \emph{$\dG$ is the thin proper 2-qBMG with color classes $U_1\cup U_2$ and $W_1\cup W_2$ where $U_1=\{1,\ldots, 4\}, U_2=\{5,\ldots, 8\}, W_1=\{9,\ldots, 12\}, W_2=\{13,\ldots, 16\}$.  $\Gamma=\aut_I(\dG)$ with $\Gamma$ the subgroup of $\pi$ ranging over all permutations on $U_1$ in (\ref{eq23102024}).}}
\label{figura7}
\end{figure}

Let $\pi$ be any permutation on $U_1$. Then $\alpha\circ \pi\circ \alpha^{-1}$ is a permutation on $W_1$, $\beta\circ \alpha\circ \pi (\beta\circ \alpha)^{-1}$ is a permutation on $U_2$, and
$\delta\circ\pi\circ\delta^{-1}$ is a permutation on $W_2$. A permutation $\varphi$ on $V$ is defined as follows:
\begin{equation}
\label{eq23102024}
\varphi(x):=
\begin{cases} 
{\mbox{$\pi(x)$ for $x\in U_1$}},\\
{\mbox{$(\alpha\circ \pi\circ \alpha^{-1})(x)$ for $x\in W_1$}},\\
{\mbox{$((\beta\circ\alpha)\circ \pi\circ (\beta\circ\alpha)^{-1})(x)$ for $x\in U_2$}},\\
{\mbox{$(\delta\circ \pi\circ \delta^{-1})(x)$ for $x\in W_2$}}.\\
\end{cases}
\end{equation}
We prove that $\varphi\in \aut_I(\dG)$. Take two vertices $u,w$ such that $uw\in E$. We have to show that then $\varphi(u)\varphi(w)\in E$. 
Assume that $u_1\in U_1,w_1\in W_1$ with $u_1w_1\in E$. Then $w_1=\alpha(u_1)$ whence $$\varphi(w_1)=\varphi(\alpha(u_1))=(\varphi\circ\alpha)(u_1)=(\alpha\circ\pi\circ\alpha^{-1})(\alpha)(u_1)=\alpha(\pi(u_1)).$$
Therefore, $\varphi(u_1)\varphi(w_1)=\pi(u_1)\alpha(\pi(u_1))$ and this shows the claim for $u_1\in U_1$ and $w_1\in W_1$.

Assume that $w_1\in W_1,u_2\in U_2$ with $w_1u_2\in E$. Then there exists $u_1\in U_1$ such that $w_1=\alpha(u_1)$. Therefore, $\varphi(w_1)=\varphi(\alpha(u_1))=(\alpha\circ\pi\circ\alpha^{-1})(\alpha(u_1))=\alpha(\pi(u_1)).$
Furthermore, $u_2=\beta(w_1)=(\beta\circ \alpha)(u_1)$. Therefore, 
$$\varphi(u_2)=(\beta\circ\alpha\circ\pi\circ\alpha^{-1}\circ\beta^{-1})(\beta(w_1))=(\beta\circ\alpha\circ\pi)(\alpha^{-1}(w_1))=
(\beta\circ\alpha\circ\pi)(u_1).$$
From this $\varphi(u_2)=(\beta\circ\alpha)(\pi(u_1))$ whence the claim follows for $w_1\in W_1,u_2\in U_2$.

Assume that $u_2\in U_2,w_2\in W_2$ with $u_2w_2\in E$. Then there exists $u_1\in U_1$ such that $u_2=(\beta\circ\alpha)(u_1)$. Thus 
$$\varphi(u_2)=\varphi((\beta\circ\alpha)(u_1))=(\beta\circ\alpha\circ\pi\circ(\beta\circ\alpha)^{-1})(\beta\circ\alpha)(u_1)=(\beta\circ\alpha)(\pi(u_1)).$$
Also, $w_2=\gamma(u_2)=(\gamma\circ\beta\circ\alpha)(u_1)$ whence
$\varphi(w_2)=(\delta\circ\pi\circ\delta^{-1})(\gamma\circ\beta\circ\alpha)(u_1).$
Since $\delta=\gamma\circ\beta\circ\alpha$ this reads
$$\varphi(w_2)=(\gamma\circ\beta\circ\alpha\circ\pi\circ\alpha^{-1}\circ\beta^{-1}\circ\gamma^{-1}\circ\gamma\circ\beta\circ\alpha)(u_1)=(\gamma\circ\beta\circ\alpha)(\pi(u_1)).$$
Therefore, $\varphi(w_2)=\delta(\pi(u_1))$. This shows that the claim follows for $u_2\in U_2$ and $w_2\in W_2$.

Finally, assume that $u_1\in U_1,w_2\in W_2$ with $u_1w_2\in E$. Then $w_2=\delta(u_1)$ and hence
$\varphi(w_2)=\varphi(\delta(u_1))=(\delta\circ\pi\circ\delta^{-1})(\delta(u_1))=\delta(\pi(u_1))$. From this the claim follows for $u_1\in U_1$ and $w_2\in W_2$.

The automorphisms $\varphi$ when $\pi$ ranges over all permutations on $U_1$ form a subgroup $\Gamma$ of $\aut_I(\dG)$ which is isomorphic to the symmetric group  ${\rm{Sym}}_m$ on $m$ letters where $m=|U_1|$. Moreover, the sets $U_1,U_2,W_1,W_2$ are $\Gamma$-orbits, and $\Gamma$ acts on each of them as ${\rm{Sym}}_m$ on $m$ letters. 
It may happen that $\Gamma=\aut_I(\dG)$, as for the 2-qBMG in Figure~\ref{figura7}.  
\end{example}
\begin{example} 
\label{ex2510204}
A slightly modified version of Example~\ref{ex2310204} provides a thin 2-qBMG which is (N2)-trivial. Again, take four pairwise disjoint sets $U_1,U_2,W_1,W_2$ of the same size $m$ consisting of positive integers; but this time label them as follows:  
\begin{equation}\label{eq23102024A}
\begin{cases}
U_1=\{1,2,\ldots,m\},\,\, W_1=\{m+1,m+2\ldots,2m\},\\ 
W_2=\{2m+1,2m+2\ldots,3m\},\,\, U_2=\{3m+1,3m+2\ldots,4m\}.
\end{cases}
\end{equation}
Choose three invertible functions, say $\alpha,\beta, \gamma$, where $$\alpha: U_1\mapsto W_1,\,\beta:=W_1 \mapsto U_2,\, \gamma:=W_2\mapsto U_2.$$
Then $\delta=\gamma^{-1}\circ\beta\circ\alpha$ is an invertible function from $U_1$ onto $W_2$.  Consider the bipartite digraph $\dG$ with color classes $U=U_1\cup U_2$ and $W=W_1\cup W_2$ where the edge set $E$ is defined as follows: for $u_1\in U_1$ and $w_1\in W_1$, $E$ contains $u_1w_1$ if and only if $w_1=\alpha(u_1)$; for $w_1\in W_1,u_2\in U_2$, $E$ contains $w_1u_2$ if and only if $u_2=\beta(w_1)$; for $u_2\in U_1,w_2\in W_2$, $E$ contains $w_2u_2$ if and only if $u_2=\gamma(w_2)$; and finally for $u_1\in U_1, w_2\in W_2$ if and only if $w_2=\delta(u_1)$. No more edge is contained in $E$. Then $\dG(V,E)$ with $V=U\cup W$ is a 2-qBMG which is (N2)-trivial but is neither (N1)- nor (N3)-trivial, as $W_1$ and $W_2$ contain only sinks. Moreover, the sources of $\dG$ are the vertices in $U_1$ whereas the sinks of $\dG$ are those in $U_2$.
Figure \ref{figura6} illustrates the case where $m=4$ and  
$$\alpha=\left[
  \begin{array}{cccc}
    1 & 2 & 3 & 4\\
    6 & 5 & 8 & 7\\
     \end{array}
\right], \quad
\beta=\left[
  \begin{array}{cccc}
    9 & 10 & 11 & 12\\
    13 & 14  & 15 & 16\\
     \end{array}
\right],
\quad
\gamma=\left[
  \begin{array}{cccc}
    9 & 10 & 11 & 12\\
    13 & 14 & 15 & 16 \\
     \end{array}
\right]
.
$$
\begin{figure}[ht]
\centering
\scalebox{0.55}
{\includegraphics{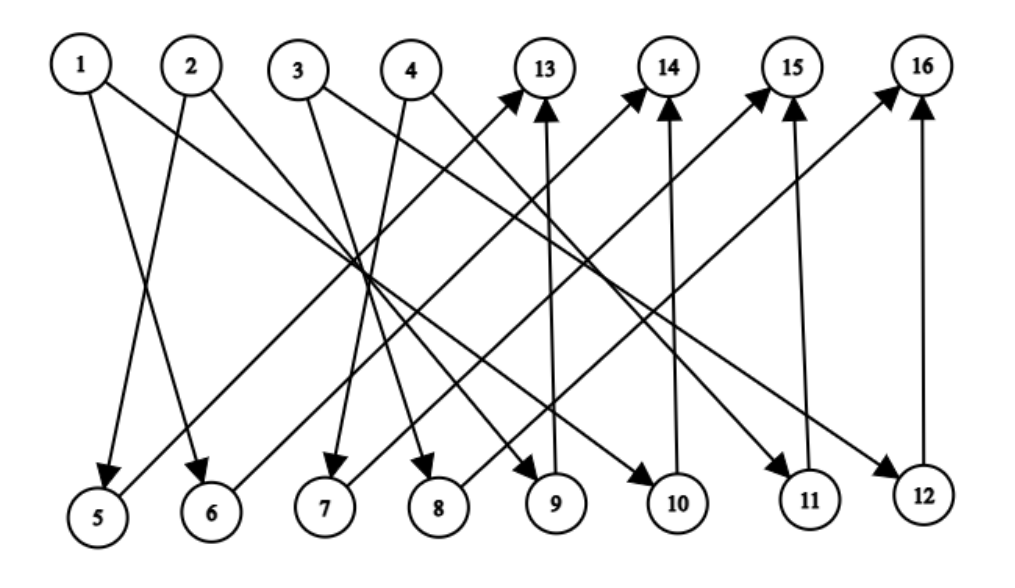}}
\caption{\emph{Case $m=4$. $\dG$ is the thin (N2)-trivial 2-qBMG with color classes $U_1\cup U_2$ and $W_1\cup W_2$ where $U_1=\{1,\ldots, 4\}, U_2=\{5,\ldots, 8\}, W_1=\{9,\ldots, 12\}, W_2=\{13,\ldots, 16\}$. The group $\Gamma$ of $\pi$ ranging over all permutations on $U_1$ in (\ref{eq23102024}) is a proper subgroup of $\aut_I(\dG)$.}}
\label{figura6}
\end{figure}

Let $\pi$ be any permutation on $U_1$. As in Example \ref{ex2310204}, every permutation $\pi$ on $U_1$ gives rise to an automorphism $\varphi\in \aut_I(\dG)$, defined as in (\ref{eq23102024}) up to replacing $\gamma$ with $\gamma^{-1}$. The final claims in Example \ref{ex2310204} about $\Gamma$ hold true for Example~\ref{ex2510204}. However, $\Gamma$ may be a proper subgroup of $\aut_I(\dG)$. This occurs for the 2-qBMG in Figure \ref{figura6} whose color-preserving automorphism group has order $384$ much larger than $24$, the order of ${\rm{Sym}}_4$, and it has $3$-orbits namely $U_1,U_2$, and $W_1\cup W_2$. 
\end{example}
\begin{example} \label{ex:general_case}
We show that the 2-qBMG obtained by Example \ref{ex2310204} may be viewed as a special case of a much more general example. We start with two families  $U$ and $V$ consisting of the same number of pairwise distinct subsets of the same size, say $U=\{U_1,\ldots, U_s\}$ and $W=\{W_1,\ldots, W_s\}$ where $|U_1|=\cdots=|U_s|=|W_1|=\cdots =|W_k|=m$. 
Take an invertible function $f_{i,i}$, for every $1\le i \le s$, from $U_i$ onto $W_i$ together with an invertible function $g_{j,j+1}$ for every $1\le j \le s-1$, from $W_j$ onto $W_{j+1}$.

We construct a bipartite digraph $\dG$ with color classes $U=U_1\cup\cdots \cup U_s$ and $W=W_1\cup \cdots \cup W_s$. The edge set $E$ of $\dG$ is constructed in several steps.

The first one is to introduce the edges $u_iw_i$ for $i=1,\ldots ,s$: If $u_i\in U_i$ and $w_i\in W_i$, then $E$ contains $u_iw_i$ 
whenever $w_i=f_{i,i}(u_i)$.

The second step adds the edges $w_{j}u_{j+1}$ for $j=1,\ldots,s-1$: If $w_j\in W_j$ and $u_{j+1}\in U_{j+1}$, then  
$w_{j}u_{j+1}\in E$ whenever  $u_{j+1}=g_{j+1,j}(w_j)$.

The third step adds the edges $u_iw_{i+1}$ for $i=1,\ldots,s-1$: If $u_i\in U_i$ and  $w_{i+1}\in W_{i+1},$ then $E$ contains $u_iw_{i+1}$ whenever $w_{i+1}=f_{i,i+1}(u_i)$ where $f_{i,i+1}$ is defined by $f_{i,i+1}=f_{i+1,i+1}\circ g_{i,i+1}\circ f_{i,i}$. 

The forth step adds the edges $u_iw_{i+2}$ for $i=1,\ldots,s-2$: If $u_i\in U_i$ and $w_{i+2}\in W_{i+2}$, then $E$ contains $u_iw_{i+2}$ whenever $w_{i+2}=f_{i,i+2}(u_i)$ where $f_{i,i+2}$ is defined by $f_{i,i+2}=f_{i+1,i+2}\circ g_{i,i+1}\circ f_{i,i}$.

The fifth step adds the edges $w_ju_{j+2}$ for $j=1,\ldots,s-2$: If $w_j\in W_j$ and $u_{j+2}\in U_{j+2}$, then $E$ contains $w_ju_{j+2}$ whenever $u_{j+2}=g_{j,j+2}(w_j)$ where $g_{j,j+2}$ is defined by $g_{j,j+2}=g_{j+1,j+2}\circ f_{j+1,j+1}\circ g_{j,j+1}$.

More generally, for $1\le i<k \le s$, let
$$f_{i,k}=f_{k,k}\circ g_{k-1,k}\circ f_{k-1,k-1}\circ \cdots \circ f_{i+1,i+1}\circ g_{i,i+1}\circ f_{i,i},$$
and, for $1\le j<d \le s$, let
$$g_{j,d}=g_{d-1,d}\circ f_{d-1,d-1}\circ g_{d-2,d-1}\circ \cdots \circ g_{j+1,j+2}\circ f_{j+1,j+1}\circ g_{j,j+1}$$
Now, we are able to define all edges of $E$: for $1\le i\le j \le s$,
if $u_i\in U_i,w_j\in W_j$ then $E$ contains $u_iw_j$ whenever $w_j=f_{i,j}(u_i)$; for $1\le j <i \le s$, if $w_j\in W_j,u_i\in U_i$ then $E$ contains $w_ju_i$ whenever $u_i=g_{i,j}(w_j)$. 

Let $\dG=\dG(U\cup W,E)$ be the above constructed bipartite graph. 
It may be noticed that the special case in Example~\ref{ex2310204} occurs for $s=2$ and $m=4$ with $U_1,U_2,W_1,W_2$, and
\begin{equation}
\label{eq23102024C}
\alpha=f_{1,1},\,\beta=g_{1,2},\,\gamma=f_{2,2},\,\delta=f_{1,2}. 
\end{equation}
Another example illustrating the above procedure is the 2-qBMG in Figure \ref{figura8} for $s,m=3$ where 
$$f_{1,1}=\left[
  \begin{array}{ccc}
    1 & 2 & 3\\
    11 & 12 & 10 \\
     \end{array}
\right], \quad
f_{2,2}=\left[
  \begin{array}{ccc}
    4 & 5 & 6\\
    13 & 15 & 14 \\
     \end{array}
\right],
\quad
f_{3,3}=\left[
  \begin{array}{ccc}
    7 & 8 & 9\\
    16 & 18 & 17 \\
     \end{array}
\right]
;
$$
$$g_{1,2}=\left[
  \begin{array}{ccc}
    10 & 11 & 12\\
    4 & 6 & 5 \\
     \end{array}
\right], \quad
g_{2,3}=\left[
  \begin{array}{ccc}
    13 & 14 & 15\\
     8 & 9  & 7 \\
     \end{array}
\right],
\quad
g_{1,3}=\left[
  \begin{array}{ccc}
    10 & 11 & 12\\
    8 & 9 & 7 \\
     \end{array}
\right]
;
$$
$$f_{1,2}=\left[
  \begin{array}{ccc}
    1 & 2 & 3\\
    14 & 15 & 13 \\
     \end{array}
\right], \quad
f_{2,3}=\left[
  \begin{array}{ccc}
    4 & 5 & 6\\
    18 & 16  & 17 \\
     \end{array}
\right],
\quad
f_{1,3}=\left[
  \begin{array}{ccc}
    1 & 2 & 3\\
    17 & 16 & 18; \\
     \end{array}
\right]
.
$$
Moreover, 
\begin{equation}
\label{eq24102024}
{\mbox{$ f_{i,k}=f_{d,k}\circ g_{j,d}\circ f_{i,j}$  for any $1\le i\le j< d \le k \le s$.}} 
\end{equation}
and
\begin{equation}
\label{eq24102024A}
{\mbox{$ g_{j,d}=g_{t,d}\circ f_{\ell,t}\circ g_{j,\ell}$  for any $1\le j <\ell \le t < d \le s$.}} 
\end{equation}
By example, $\dG(V,E)$ is a thin bipartite graph with color classes $U$ and $V$. Moreover, (\ref{eq24102024}) and (\ref{eq24102024A}) show that $\dG(U\cup W,E)$ has properties (N1) and (N2). Also, if $u_i\in U_i, u_k\in U_k$ with $i<k$ and $N^+(u_i)\cap N^+(u_k)\ne \emptyset$ then $N^+(u_k)\subset N^+(u_i)$. Similarly, if $w_j\in W_j, w_k\in U_k$ with $j<k$ and $N^+(w_j)\cap N^+(w_k)\ne \emptyset$ then $N^+(w_k)\subset N^+(w_j)$, as $|N^+(v)|\leq 1$ for all $v\in V$. Hence also (N3) holds for $\dG$. Therefore, $\dG(U\cup W,E)$ is a proper 2-qBMG.
\begin{figure}[ht]
\centering
\scalebox{0.55}
{\includegraphics{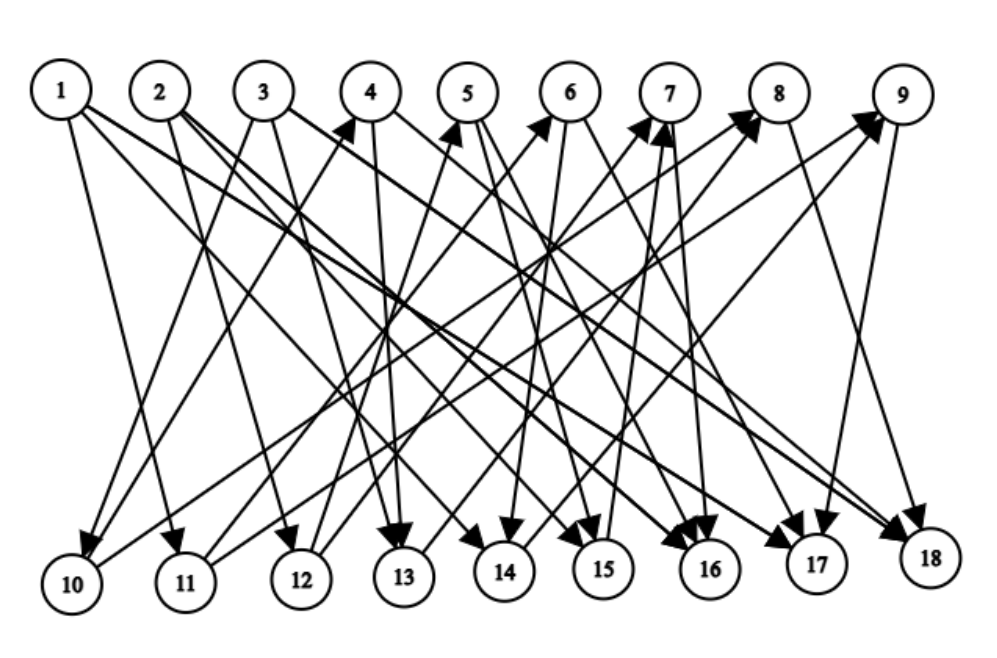}}
\caption{\emph{Case $s=3=m$. $\dG$ is a proper 2-qBMG with color classes $U_1\cup U_2 \cup U_3$ and $W_1\cup W_2 \cup W_3$ where $U_1=\{1,2,3\}, U_2=\{4,5,6\}, U_2=\{7,8,9\}, W_1=\{10,11,12\}, W_2=\{13,14,15\}, W_3=\{16,17,18\}$. The group $\Gamma$ of $\pi$ ranging over all permutations on $U_1$ in (\ref{eq23102024}) is a proper subgroup of $\aut_I(\dG)$.}}
\label{figura8}
\end{figure}

Every permutation $\pi$ on $U_1$ gives rise to an automorphism $\varphi$ of $\dG$. The arguments used in Example~\ref{ex2310204} can be adapted to show that 
$\varphi\in \aut_I(\dG)$ where
$$\varphi(x):=
\begin{cases} 
{\mbox{$\pi(x)$ for $x\in U_1$,}}\\
{\mbox{$(f_{1,j}\circ \pi\circ f_{1,j}^{-1})(x)$ for $x\in W_j$ and $1\le j\le s$}},\\ 
{\mbox{($g_{1,i}\circ f_{1,1}\circ \pi\circ (g_{1,i}\circ f_{1,1})^{-1})(x)$ for $x\in U_i$ and $1 \le j \le s$}}.
\end{cases}
$$ 

As in Example~\ref{ex2310204}, the automorphisms $\varphi$ with $\pi$ ranging over all permutations on $U_1$ form a subgroup $\Gamma\cong {\rm{Sym}}_m$ of $\aut_I(\dG)$, such that the sets $U_i$ and $W_j$ are $\Gamma$-orbits, and that $\Gamma$ acts on each them as ${\rm{Sym}}_m$ on $m$ letters.
\end{example}
Therefore, Example~\ref{ex:general_case} gives the following result. 
\begin{theorem}
\label{the2410204} For any two integers $m,s\ge 2$ there exists a proper, thin 2-qBMG on $2ms$ vertices whose color-preserving automorphism group contains a subgroup isomorphic to ${\rm{Sym}}_m$.  \end{theorem}
\section{Orientations and automorphisms of 2-qBMGs}
\label{oa}
We state two properties of orientations of 2-qBMGs.
\begin{theorem}
\label{thm:orientation2qBMG} 
Every orientation of a 2-qBMG satisfying ($*$) is a 2-qBMG. 
\end{theorem}
\begin{proof} Let $\dG=\dG(V,E)$ be a 2-qBMG satisfying ($*$). Take an orientation $\overrightarrow{O}=\overrightarrow{O}(V,F)$ of $\dG$ with $F\subseteq E$. We prove that $\overrightarrow{O}$ also satisfies (N1),(N2) and (N3*). 

Take an (N1)-configuration [u,t,w,v] in $\dO$. This is also an (N1)-configuration in $\dG$, therefore $uv$ are dependent in $E$ and also in $F$. Hence, (N1) holds for $\dO$. 

To show (N2) take $u,v,w,t\in V$ such that $uv,vw,wt\in F$. As $F\subseteq E$, and (N2) holds in $\dG$, this yields $ut\in E$. If $ut$ is not a symmetric edge of $\dG$, then $ut\in F$ also holds. If $ut$ is a symmetric edge of $\dG$ then $tu,uv,vw\in E$ and hence $tw\in E$ from (N2). But then $t$ is an endpoint of two distinct symmetric edges, namely $tu$ and $tw$, a contradiction with ($*$).

To show that (N3) holds take $u,v\in V$ such that $w\in N(u)\cap N(v)$ in $\overrightarrow{O}$. Then (N3) applies to $u$ and $v$, and $N^+(u)\subseteq N^+(v)$ or $N^+(v)\subseteq N^+(u)$ in $\dG$. Suppose this does not hold for $\dO$. Then there are $v^*\in  N^+(u) \setminus N^+(v)$ and $u^*\in  N^+(v) \setminus N^+(u)$ in $\dO$. Hence, $uu^*$ and $vv^*$ must be symmetric edges in $\dG$. From ($*$) $wv\notin E$ and $wu\notin E$, and (N3*) applies to $u$ and $v$ in $\dG$. Hence $N^-(u)=N^-(v)$ in $\dG$, and $u^*\in N^-(v)$, violating $(*)$ for $v$.
\end{proof}

If $\overrightarrow{O}$ is an orientation and $g\in \aut_I(\dG)$ is an automorphism of $\dG$, then $g$ is not an automorphism of $\overrightarrow{O}$ in general but some important exceptions occur. Such an exception is the $UW$-\emph{orientation,} which is the special orientation of $\dG$ obtained by retaining that edge from each symmetric edge whose tail and head are in the color classes $U$ and $W$, respectively. This claim is a consequence of the following theorem.     
\begin{theorem}
\label{th06112024} Let $\dG$ be a 2-qBMG with color classes $U$ and $W$. Let  $\overrightarrow{O}$ be a $UW$-orientation of $\dG$. 
Then $\aut_I(\dG)=\aut_I(\overrightarrow{O})$. 
\end{theorem}
\begin{proof} Take a symmetric edge with endpoints $u$ and $w$. Then $uw,wu\in E$.

If $g\in\aut_I(\dG)$ then $u\in U$ implies $g(u)\in U$, and hence $g(w)\in W$. Therefore, $g$ takes the edge $uw$ of $\overrightarrow{O}$ to the edge $g(u)g(w)$, which is also an edge in $\overrightarrow{O}$. Thus $\aut_I(\dG)\le  \aut_I(\overrightarrow{O})$. 

Conversely, let $g\in \aut_I(\overrightarrow{O})$, and extend the action of $g$ to $E$ by defining the image of $wu\in E$ by $g(wu)=g(w)g(u)$. Then $uw,wu\in E$ implies that $g(u)g(w)\in E$ is a symmetric edge. Therefore, $g(w)g(u)\in E$. Thus $\aut_I(\overrightarrow{O})\le \aut_I(\dG)$.  
\end{proof}

\bibliographystyle{unsrtnat}
\bibliography{ref}

\end{document}